\documentclass[12pt]{amsart}
\usepackage[utf8]{inputenc}
\usepackage{amsmath,amssymb,amsthm,mathtools}
\usepackage{tikz}
\usepackage{graphicx}

\theoremstyle{plain}
\newtheorem{theorem}{Theorem}

\renewcommand{\vec}[1]{\mathbf{#1}}

\begin{document}
\author{Iwan Praton}
\address{Franklin \& Marshall College,
Lancaster, PA 17604}
\email{ipraton@fandm.edu}
\title{Tilings with the Minimal Tile Property}
\begin{abstract}
A square tiling of the unit square is said to have the
minimal tile property if the smallest tile can tile
all the other tiles. We show that in such a tiling,
the smallest tile cannot be too small.
\end{abstract}

\maketitle

Start with a unit square and fix a positive integer $n$.
Tile the unit square with $n$ smaller squares. (In other words,
put $n$ non-overlapping small squares---which
do not have to be the same size---inside the unit square,
with no space left over). This can be done for $n=4$
and $n\geq 6$. There are many questions
we can ask about this configuration, the most famous of which
is probably the ``squaring the square'' problem \cite{BSST}.
 Here we focus
on a different question. Following Erd\H{o}s and
Soifer, we ask how large
the sum of the sides of the small squares can be. There are
good conjectures for the answer \cite{ErdosSoifer}, 
but the problem appears
difficult. In order to make the problem tractable to 
computers, Alm \cite{Alm} introduced a new additional condition:
we consider only tilings that have the minimal tile
property (MTP): the smallest square tile is required to tile all
the other tiles. Thus the side lengths of all the tiles
have to be an integer multiple of the side length
of the smallest tile. 

It is not immediately obvious that MTP tilings exist
for all values of $n$ for which regular tilings exist.
But they do, as shown in the figure below.

\begin{figure}[h]
\begin{center}
\begin{tikzpicture}[scale=0.4]
\draw[<-] (0,8)--(2.5,8);\node at (3.5,8){$k$};
\draw[->] (4.5,8)--(7,8);
\draw[thick] (0,0) rectangle (7,7);
\draw[thick] (0,0) rectangle (6,6);
\draw[thick] (6,6) rectangle (7,7);
\draw[thick] (0,6) rectangle (1,7);
\draw[thick] (5,6) -- (5,7);
\node at (3.5,6.5) {$\cdots$};
\draw[thick] (6,0) rectangle (7,1);
\draw[thick] (6,5) -- (7,5);
\node at (6.5,3.5) {$\vdots$};
\end{tikzpicture}
\qquad
\begin{tikzpicture}[scale=0.4]
\draw[<-] (0,8)--(2.5,8);\node at (3.5,8){$k$};
\draw[->] (4.5,8)--(7,8);
\draw[thick] (0,0) rectangle (7,7);
\draw[thick] (0,0) rectangle (6,6);
\draw[thick] (6,6) rectangle (7,7);
\draw[thick] (0,6) rectangle (1,7);
\draw[thick] (5,6) -- (5,7);
\node at (3.5,6.5) {$\cdots$};
\draw[thick] (6,0) rectangle (7,1);
\draw[thick] (6,5) -- (7,5);
\node at (6.5,3.5) {$\vdots$};
\draw[thick] (6.5,6)--(6.5,7);
\draw[thick] (6,6.5) -- (7,6.5);
\end{tikzpicture}
\caption{MTP tilings for $n=2k$ and $n=2k+3$, where
$k\ge 2$.}
\end{center}
\end{figure}
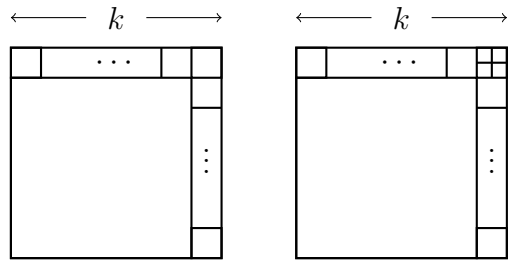

The main aim of MTP tilings is to use a computer to help check
all possibilities. We thus need to show that there are only
finitely many possibilities to check. One way to do this
is to show that for a given value of $n$, the smallest tile 
cannot be too small. Then there are only finitely many
possible values for the side lengths of all tiles, so we
only need to check a finite number of configurations.
The theorem below is a result of this nature.

\begin{theorem}
In an MTP tiling with $n$ squares, the side length of
the smallest tile is at least $1/2^n$.
\end{theorem}

The proof uses the electrical network theory of Brooks et al
\cite{BSST}. It turns out to be convenient to adjust our scale so that
the smallest tile has side length $1$. Then all tiles
in an MTP tiling has integer side lengths; the side
length of the big tiled square is also an integer. 

We now summarize the electrical network approach
to analyzing tilings. Details are in \cite{BSST} or \cite{Wagner}. 
Suppose
we have a rectangle of height $h$ and width $w$ tiled
by squares whose side lengths are integers. We place
the rectangle and its tiles so that all sides are
horizontal or vertical. First we convert the tiling into a graph
$G(V,E)$, where the vertex set $V$ consists of all 
connected components of the horizontal boundaries of
the tiling, and the edge set $E$ is the set of all
squares in the tiling. Since every tile has two horizontal
boundaries, it is natural to say that two vertices is 
connected by an edge if the corresponding tile has the
two corresponding horizontal components as boundaries.
(Note that it is possible to have multiples edges connecting
two vertices.) 

Let $m$ denote the number of vertices in the graph $G$.
We label the vertices of $G$ with integers $1$ to $m$.
The vertex corresponding
to the top edge of the rectangle is vertex 1; similarly,
the vertex corresponding to the bottom edge
is vertex $m$. 

In the simple example below, the rectangle has height 2
and width 4, tiled with 5 squares.
The corresponding graph is shown on the right. 

\begin{figure}[h]
\begin{center}
\begin{tikzpicture}
\draw (0,0) -- (4,0);
\draw (0,1) -- (1,1);
\draw (3,1) -- (4,1);
\draw (0,2) -- (4,2);
\draw (0,0) -- (0,2);
\draw (1,0) -- (1,2);
\draw (3,0) -- (3,2);
\draw (4,0) -- (4,2);
\draw (3.5,1)--(3.5,2);
\draw (3,1.5)--(4,1.5);
\filldraw (8,2) circle (3pt);
\filldraw (7,1) circle (3pt);
\filldraw (9,1) circle (3pt);
\filldraw (8,0) circle (3pt);
\filldraw (8.5,1.5) circle (3pt);
\draw (9,1) -- (8,0)--(7,1)--(8,2)--(8,0);
\draw (8,2) to [out=0,in=90] (8.5,1.5);
\draw (8.5,1.5) to [out=-90,in=180] (9,1);
\draw (8.5,1.5) to [out=0,in=90] (9,1);
\draw (8,2) to [out=-90,in=180] (8.5,1.5);
\node [right] at (7.5,2) {$1$};
\node [right] at (8.1,0) {$5$};
\end{tikzpicture}
\caption{A rectangle tiled by squares and
its corresponding network graph.}
\end{center}
\end{figure}
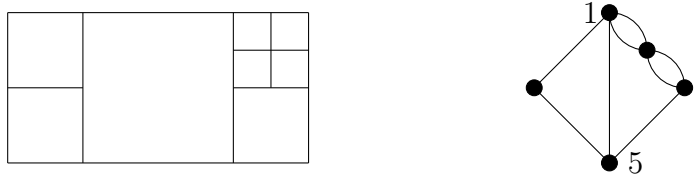

We now think of $G$ as an electric network. The edges
are wires through which electric currents can flow. We'll
assume that each edge has unit resistance. Let's say
that one unit of current flows into vertex 1, goes through
the network, and leaves from vertex $m$ (which is
grounded, i.e., vertex $m$ has zero electric potential).
 Then
it is possible, using Kirchoff's and Ohm's Laws, to calculate
the amount of current in each wire and the electric
potential in each vertex. The calculations can be 
conveniently summarized in matrix form. 

Define an $m\times m$ matrix $L=(a_{ij})$,
where $a_{ii}$ is the degree of vertex $i$, and for $i\neq j$,

\[a_{ij}=\begin{cases}
0 & \text{if vertices $i$ and $j$ are not connected;}\\
 -r & \text{if there are $r$ edges connecting $i$ and $j$.}
\end{cases}
\]

It turns out that $L$ is a singular matrix with rank $m-1$.
For our example, the matrix $L$ is
\[
L=\begin{pmatrix*}[r]
4 & -2 & 0 & -1 & -1\\
-2 & 4 & -2 & 0 & 0\\
0 & -2 & 3 & 0 & -1\\
-1 & 0 & 0 & 2 & -1\\
-1 & 0 & 2 & -1 & 3\\
\end{pmatrix*}.
\]

Now let $\vec{p}=(p_1,\ldots,p_m)^T$, where $p_i$
is the electric potential at vertex $i$. Note that $p_m=0$
since vertex $m$ is grounded. Ohm's Law 
indicates that the $i$th component of $L\vec{p}$
is the total amount of current flowing into (and out of)
vertex $i$; by Kirchoff's Law (which is a current
preservation law), this amount is 
$1$ when $i=1$;  $-1$ when $i=m$;
and $0$ when $i\neq 1,m$.
Therefore $L\vec{p}=\vec{e_1}-\vec{e}_m$ where
$e_i$ is the $i$th standard basis vector. Thus to
find $p_k$ we need to solve the matrix equation
$L\vec{v}=\vec{e_1}-\vec{e}_m$; we want the
unique solution with $v_m=0$.

We can get this solution by first solving the equation
$L'\vec{v}'=\vec{e}_1$, where $L'$ is the matrix
$L$ without its last row and last column. The solution
we want is $\vec{v}'$ with a $0$ appended as the
last entry. The matrix $L'$ is invertible, so there
is a unique solution.

In our example, $L'=\begin{psmallmatrix*}[r]
4 & -2 & 0 & -1 \\
-2 & 4 & -2 & 0\\
0 & -2 & 3 & 0 \\
-1 & 0 & 0 & 2\\
\end{psmallmatrix*}$,
whose determinant is $32$. The solution to
$L'\vec{v}'=\vec{e}_1$ is $\vec{v}'=\frac18(4,3,2,2)^T$, so
the vector of potentials is $\vec{p}=\frac18(4,3,2,2,0)^T$.

We can also find the potentials a different way.
Go back to the tiling of the rectangle. Place
the bottom edge of the rectangle at $y=0$ and let 
$q_i$ denote the $y$-coordinate of the horizontal
boundary that corresponds to the vertex $i$.
Define the vector 
$\vec{q}$ as $(q_1,\dots,q_m)^T$. 
We note that $q_m=0$, $q_1=h$, the height of the rectangle,
and each $q_i$ is an integer. 
In our example, $\vec{q}=(4,3,2,2,0)$. 

We define potentials and currents in $G$ using $\vec{q}$
as follows; we'll show that these potentials and currents
satisfy Ohm's and Kirchoff's Laws.
We define the electric potential at vertex $i$ to be $q_i$.
Previously we imposed 1 unit of current flowing through
$G$ from vertex $1$ to vertex $m$, but now 
instead of 1 unit we will use $w$ units, where $w$
is the width of the rectangle. The currents in the interior
of $G$ are defined as follow.

If there is an edge connecting vertex $i$ and vertex $j$,
we say there is a current  of $q_i-q_j$ units
flowing from $i$ to $j$. These
currents and potentials clearly satisfy Ohm's Law (since
each edge has unit resistance). Now suppose
vertex $j$ corresponds to an interior horizontal
boundary. A tile above and on this
horizontal boundary contributes $q_i-q_j$,
the side length of the tile, to the current coming into vertex $j$.
The total amount from such tiles is the sum of the
side lengths of these tiles, so it is equal to the length
of the boundary. Similarly, the total contribution
of the tiles on and below the horizontal boundary
the negative of the length of this boundary. 
Therefore the total current coming into and out of
vertex $j$ is zero, in accordance to Kirchoff's Law.
Of course, these laws are also satisfied at vertex
$1$ and $m$. Therefore the potentials $q_i$
and the currents $q_i-q_j$ are exactly the values
attached to the electrical network, except that
the external current from vertex $1$ to vertex $m$
is now $w$ units. 

We conclude that
\[
L\vec{q}=w\vec{e}_1-w\vec{e}_m,
\]
and so $L(\vec{q}/w)=\vec{e}_1-\vec{e}_m$. 
Since potentials are unique, we have $\vec{p}=\vec{q}/w$,
providing an explicit solution to the problem
of finding the electric potentials of the network. 

This is sufficient to provide a proof of Theorem 1.

\begin{proof}[Proof of Theorem 1]
As above, we scale our tiling so that the smallest tile
has side length $1$. It suffices to prove that the side
length of the big tiled square is at most $2^n$.
In the notation above, we want to show that $h=w\leq 2^n$.

First we note that in this case, there is a tile of side
length $1$, so $q_1,\ldots,q_m$ are integers with no
common factors (other than $1$). 
Let $d=\det L'$. Then $d\vec{p}$ is a vector of integers
that is a scalar multiple of $\vec{q}$. Since the
components of $\vec{q}$ are relatively prime, in fact $d\vec{p}$ is an
integer multiple of $\vec{q}$. In particular, $dp_1>0$ is
an integer multiple of $q_1>0$, so $dp_1\geq q_1$.
Note that $q_1=h$ and $p_1=q_1/w=h/w$, so $dh/w\geq h$ or
$w\geq d$.

By the Matrix-Tree theorem \cite{Wagner}, $d$ is the number of spanning
trees of the graph $G$. Since a spanning tree is a subset
of the edges of $G$, we have the crude estimate $d\leq 2^n$,
where $n$ is the number of edges of $G$, i.e., the number
of squares in the tiling. Thus we have $w\leq 2^n$, which
is what we want. 
\end{proof}

Note that the proof does not rely on the minimal tile
property: what we need is that the components of $\vec{q}$
are relatively prime. It is straightforward to verify that
this can be realized for any tiling of the unit square.
Thus a similar result holds for all tilings, not just MTP tilings.

Of course, the bound here is quite crude and not useful
for computational purposes unless $n$ is small. For
specific values of $n$ there are better bounds.

Here is an example. First we introduce some notation.
If $T$ is a tiling of the unit square with $n$ tiles, 
denote by 
$\sigma(T)$ the sum of the side lengths of the tiles
in $T$. We want to find an optimal tiling $T_M$ such that
$\sigma(T_M)\geq \sigma(T)$ for all MTP tilings
$T$. More generally, we would like to find $T_M$
such that $\sigma(T_M)\geq \sigma(T)$ for all tilings
$T$, not just MTP tilings. 

Suppose now $n=k^2+3$ where $k$ is
a positive integer. The standard $k\times k$ grid,
where one tile is further divided into 4 equal pieces,
is conjectured to be optimal over all tilings
\cite{ErdosSoifer}. It also happens to be an MTP 
tiling. The figure below shows an example
with $k=3$.

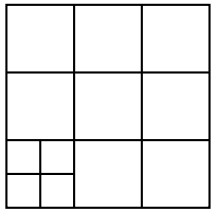
\begin{figure}[h]
\begin{center}
\begin{tikzpicture}[scale=0.3]
\draw[thick] (0,0) rectangle (9,9);
\draw[thick] (0,3) -- (9,3);
\draw[thick] (0,6) -- (9,6);
\draw[thick] (3,0) -- (3,9);
\draw[thick] (6,0) -- (6,9);
\draw[thick] (1.5,0)--(1.5,3);
\draw[thick] (0,1.5)--(3,1.5);
\end{tikzpicture}
\caption{MTP tiling for $n=k^2+3$ (here $k=3$).}
\end{center}
\end{figure}

The sum of the side lengths of this MTP tiling is
$k+1/k$. Thus $\sigma(T_M)\geq k+1/k$. 

\begin{theorem}
Suppose $T_M$ is an optimal tiling with
$n=k^2+3$ (where $k\geq 2$). Then
the smallest tile in $T_M$ has side length at least
$1/(3k^3)$.
\end{theorem}
\begin{proof}
We recall a result from \cite{StatonTyler}: if 
we have $m$ nonoverlapping tiles with total area $A$,
then the sum of their side lengths is at most
$\sqrt{mA}$. 

Now let $x$ denote the length of the smallest
tile in a tiling $T$. Then its area is $x^2$ and the other
$k^2+2$ tiles have area $1-x^2$. 
Thus
\[
\sigma(T)\leq x+\sqrt{(1-x^2)(k^2+2)}
< x+\sqrt{k^2+2}<x+k+1/k-1/(3k^3),
\]
so if $x<1/(3k^3)$, then $T$ is not optimal.
\end{proof}

\end{document}